\documentclass{amsart}
\usepackage{amssymb}
\usepackage{amsfonts}
\usepackage{hyperref}
\usepackage[numbers,sort&compress]{natbib}

\newtheorem{theorem}{Theorem}
\theoremstyle{plain}

\newtheorem{corollary}{Corollary}

\newtheorem{lemma}{Lemma}

\newtheorem{remark}{Remark}

\numberwithin{equation}{section}

\input{tcilatex}

\begin{document}
\title[An accurate approximation formula for gamma function]{An accurate
approximation formula for gamma function}
\author{Zhen-Hang Yang and Jing-Feng Tian*}
\address{Zhen-Hang Yang, College of Science and Technology\\
North China Electric Power University, Baoding, Hebei Province, 071051, P.
R. China and Department of Science and Technology, State Grid Zhejiang
Electric Power Company Research Institute, Hangzhou, Zhejiang, 310014, China}
\email{yzhkm\symbol{64}163.com}
\address{Jing-Feng Tian\\
College of Science and Technology\\
North China Electric Power University\\
Baoding, Hebei Province, 071051, P. R. China}
\email{tianjf\symbol{64}ncepu.edu.cn}
\date{October 8, 2017}
\thanks{*Corresponding author: Jing-Feng Tian, e-mail: tianjf\symbol{64}%
ncepu.edu.cn}

\begin{abstract}
In this paper, we present a very accurate approximation for gamma function:%
\begin{equation*}
\Gamma \left( x+1\right) \thicksim \sqrt{2\pi x}\left( \dfrac{x}{e}\right)
^{x}\left( x\sinh \frac{1}{x}\right) ^{x/2}\exp \left( \frac{7}{324}\frac{1}{%
x^{3}\left( 35x^{2}+33\right) }\right) =W_{2}\left( x\right)
\end{equation*}%
as $x\rightarrow \infty $, and prove that the function $x\mapsto \ln \Gamma
\left( x+1\right) -\ln W_{2}\left( x\right) $ is strictly decreasing and
convex from $\left( 1,\infty \right) $ onto $\left( 0,\beta \right) $, where%
\begin{equation*}
\beta =\frac{22\,025}{22\,032}-\ln \sqrt{2\pi \sinh 1}\approx 0.00002407.
\end{equation*}
\end{abstract}

\maketitle

%\date{October 5, 2016}

\section{Introduction}

The Stirling's formula%
\begin{equation}
n!\thicksim \sqrt{2\pi n}n^{n}e^{-n}  \label{S}
\end{equation}%
for $n\in \mathbb{N}$ has important applications in probability theory,
statistical physics, number theory, combinatorics and other related fields.
There are many refinements for the Stirling's formula, see for example,
Burnside's \cite{Burnside-MM-46-1917}, Gosper \cite{Gosper-PNAS-75-1978},
Batir \cite{Batir-P-27(1)-2008}, Mortici \cite{Mortici-CMI-19(1)-2010}.

Since gamma function $\Gamma \left( x\right) =\int_{0}^{\infty
}t^{x-1}e^{-t}dt$ for $x>0$ is a generalization of factorial function, some
authors also pay attention to find various better approximations for the
gamma function, for instance, Ramanujan \cite[P. 339]{Ramanujan-SB-1988},
Windschitl \cite[Eq. (42)]{Smith-2006}, \cite{http/gamma}, Smith \cite[Eq.
(42)]{Smith-2006}, Nemes \cite[Corollary 4.1]{Nemes-AM-95-2010}, \cite%
{http/gamma}, Yand and Chu \cite[Propositions 4 and 5]{Yang-AMC-270-2015},
Chen \cite{Chen-JNT-164-2016}.

More results involving the approximation formulas for the factorial or gamma
function can be found in \cite%
{Shi-JCAM-195-2006,Guo-JIPAM-9(1)-2008,Batir-AM-91-2008,Mortici-AM-93-2009-1, Mortici-MMN-11(1)-2010,Mortici-CMA-61-2011,Zhao-PMD-80(3-4)-2012, Mortici-MCM-57-2013,Feng-NA-64-2013,Qi-JCAM-268-2014,Lu-RJ-35(1)-2014, Yang-AMC-270-2015,Lu-AMC-253-2015,Yang-JIA-2018}
and the references cited therein. Several nice inequalities between gamma
function and the truncations of its asymptotic series can be found in \cite%
{Alzer-MC-66-1997,Yang-JMAA-441-2016}.

Now let us focus on the Windschitl's approximation formula (see \cite[Eq.
(42)]{Smith-2006}, \cite{http/gamma}) defined by%
\begin{equation}
\Gamma \left( x+1\right) =\sqrt{2\pi x}\left( \dfrac{x}{e}\right) ^{x}\left(
x\sinh \frac{1}{x}\right) ^{x/2}:=W_{0}\left( x\right) .  \label{W0}
\end{equation}%
As shown in \cite{Chen-JNT-164-2016}, the rate of Windschitl's approximation 
$W_{0}\left( x\right) $ converging to $\Gamma \left( x+1\right) $ is like $%
x^{-5}$ as $x\rightarrow \infty $, and is faster if replacing $W_{0}\left(
x\right) $ by%
\begin{equation}
W_{1}\left( x\right) =\sqrt{2\pi x}\left( \dfrac{x}{e}\right) ^{x}\left(
x\sinh \frac{1}{x}+\frac{1}{810x^{6}}\right) ^{x/2}  \label{W1}
\end{equation}%
(see \cite{http/gamma}). These show that $W_{0}\left( x\right) $ and $%
W_{1}\left( x\right) $ are excellent approximations for gamma function.

In 2009, Alzer \cite{Alzer-PRSE-139A-2009} proved that for all $x>0$,%
\begin{equation}
\sqrt{2\pi x}\left( \dfrac{x}{e}\right) ^{x}\left( x\sinh \frac{1}{x}\right)
^{x/2}\left( 1+\frac{\alpha }{x^{5}}\right) <\Gamma \left( x+1\right) =\sqrt{%
2\pi x}\left( \dfrac{x}{e}\right) ^{x}\left( x\sinh \frac{1}{x}\right)
^{x/2}\left( 1+\frac{\beta }{x^{5}}\right)  \label{A-I}
\end{equation}%
with the best possible constants $\alpha =0$ and $\beta =1/1620$. Lu, Song
and Ma \cite{Lu-JNT-140-2014} extended Windschitl's formula as%
\begin{equation*}
\Gamma \left( n+1\right) \thicksim \sqrt{2\pi n}\left( \dfrac{n}{e}\right)
^{n}\left[ n\sinh \left( \frac{1}{n}+\frac{a_{7}}{n^{7}}+\frac{a_{9}}{n^{9}}+%
\frac{a_{11}}{n^{11}}+\cdot \cdot \cdot \right) \right] ^{n/2}
\end{equation*}%
with $a_{7}=1/810,a_{9}=-67/42525,a_{11}=19/8505,...$. An explicit formula
for determining the coefficients of $n^{-k}$ ($n\in \mathbb{N}$) was given
in \cite[Theorem 1]{Chen-AMC-245-2014} by Chen. Another asymptotic expansion%
\begin{equation}
\Gamma \left( x+1\right) \thicksim \sqrt{2\pi x}\left( \dfrac{x}{e}\right)
^{x}\left( x\sinh \frac{1}{x}\right) ^{x/2+\sum_{j=0}^{\infty }r_{j}x^{-j}}%
\text{, \ }x\rightarrow \infty  \label{g-ae-C0}
\end{equation}%
was presented in the same paper \cite[Theorem 2]{Chen-AMC-245-2014}.

Motivated by the above comments, the aim of this paper is to provide a more
accurate Windschitl type approximation:%
\begin{equation}
\Gamma \left( x+1\right) \thicksim \sqrt{2\pi x}\left( \dfrac{x}{e}\right)
^{x}\left( x\sinh \frac{1}{x}\right) ^{x/2}\exp \left( \frac{7}{324}\frac{1}{%
x^{3}\left( 35x^{2}+33\right) }\right) =W_{2}\left( x\right)  \label{W2}
\end{equation}%
as $x\rightarrow \infty $. Our main result is contained in the following
theorem.

\begin{theorem}
\label{MT-1}The function%
\begin{equation*}
f_{0}\left( x\right) =\ln \Gamma \left( x+1\right) -\ln \sqrt{2\pi }-\left(
x+\frac{1}{2}\right) \ln x+x-\frac{x}{2}\ln \left( x\sinh \frac{1}{x}\right)
-\frac{7}{324}\frac{1}{x^{3}\left( 35x^{2}+33\right) }
\end{equation*}%
is strictly decreasing and convex from $\left( 1,\infty \right) $ onto $%
\left( 0,f_{0}\left( 1\right) \right) $, where%
\begin{equation*}
f_{0}\left( 1\right) =\frac{22\,025}{22\,032}-\ln \sqrt{2\pi \sinh 1}\approx
0.00002407.
\end{equation*}
\end{theorem}

\section{Lemmas}

Since the function $f_{0}\left( x\right) $ contains gamma and hyperbolic
functions, it is very hard to deal with its monotonicity and convexity by
usual approaches. For this, we need the following lemmas, which provide a
new way to prove our result.

\begin{lemma}
\label{L-dpsi>}The inequality%
\begin{equation*}
\psi ^{\prime }\left( x+\frac{1}{2}\right) >x\frac{x^{4}+\frac{227}{66}x^{2}+%
\frac{4237}{2640}}{x^{6}+\frac{155}{44}x^{4}+\frac{329}{176}x^{2}+\frac{375}{%
4928}}
\end{equation*}%
holds for $x>0$.
\end{lemma}

\begin{proof}
Let%
\begin{equation*}
g_{1}\left( x\right) =\psi ^{\prime }\left( x+\frac{1}{2}\right) -x\frac{%
x^{4}+\frac{227}{66}x^{2}+\frac{4237}{2640}}{x^{6}+\frac{155}{44}x^{4}+\frac{%
329}{176}x^{2}+\frac{375}{4928}}
\end{equation*}

Then by the recurrence formula \cite[p. 260, (6.4.6)]{Abramowttz.1972}%
\begin{equation*}
\ \psi ^{\prime }\left( x+1\right) -\psi ^{\prime }\left( x\right) =-\frac{1%
}{x^{2}}
\end{equation*}%
we have%
\begin{eqnarray*}
g_{1}\left( x+1\right) -g_{1}\left( x\right) &=&\psi ^{\prime }\left( x+%
\frac{3}{2}\right) -\frac{\left( x+1\right) \left( \left( x+1\right) ^{4}+%
\frac{227}{66}\left( x+1\right) ^{2}+\frac{4237}{2640}\right) }{\left(
x+1\right) ^{6}+\frac{155}{44}\left( x+1\right) ^{4}+\frac{329}{176}\left(
x+1\right) ^{2}+\frac{375}{4928}} \\
&&-\psi ^{\prime }\left( x+\frac{1}{2}\right) +\frac{x\left( x^{4}+\frac{227%
}{66}x^{2}+\frac{4237}{2640}\right) }{x^{6}+\frac{155}{44}x^{4}+\frac{329}{%
176}x^{2}+\frac{375}{4928}}
\end{eqnarray*}%
\begin{eqnarray*}
&=&-58\,982\,400\left( 2x+1\right) ^{-2}\left(
4928x^{6}+17\,360x^{4}+9212x^{2}+375\right) ^{-1}\times \\
&&\left(
4928x^{6}+29\,568x^{5}+91\,280x^{4}+168\,000x^{3}+187\,292x^{2}+117\,432x+31%
\,875\right) ^{-1}<0.
\end{eqnarray*}%
It then follows that%
\begin{equation*}
g_{1}\left( x\right) >g_{1}\left( x+1\right) >...>\lim_{n\rightarrow \infty
}g_{1}\left( x+n\right) =0,
\end{equation*}%
which proves the desired inequality, and the proof is done.
\end{proof}

\begin{lemma}
\label{L-1/sh>}The inequalities%
\begin{equation}
\frac{t}{\sinh t}>1-\frac{1}{6}t^{2}+\frac{7}{360}t^{4}-\frac{31}{15\,120}%
t^{6}+\frac{127}{604\,800}t^{8}-\frac{73}{3421\,440}t^{10}>0  \label{t/sht>}
\end{equation}%
hold for $t\in (0,1]$.
\end{lemma}

\begin{proof}
It was proved in \cite[Theorem 1]{Yang-JMAA-441-2016} that for integer $%
n\geq 0$, the double inequality%
\begin{equation}
-\sum_{i=0}^{2n+1}\frac{2\left( 2^{2i-1}-1\right) B_{2i}}{\left( 2i\right) !}%
t^{2i-1}<\frac{1}{\sinh t}<-\sum_{i=0}^{2n}\frac{2\left( 2^{2i-1}-1\right)
B_{2i}}{\left( 2i\right) !}t^{2i-1}  \label{csch-b}
\end{equation}%
holds for $x>0$. Taking $n=2$ yields 
\begin{equation*}
\frac{1}{\sinh t}>\frac{1}{t}-\frac{1}{6}t+\frac{7}{360}t^{3}-\frac{31}{%
15\,120}t^{5}+\frac{127}{604\,800}t^{7}-\frac{73}{3421\,440}t^{9}:=\frac{%
h\left( t\right) }{t},
\end{equation*}%
which is equivalent to the first inequality of (\ref{t/sht>}) for all $t>0$.

Since $x\in (0,1]$, making a change of variable $t^{2}=1-x\in (0,1]$ we
obtain%
\begin{eqnarray*}
h\left( t\right) &=&\frac{73}{3421\,440}x^{5}+\frac{12\,371}{119\,750\,400}%
x^{4}+\frac{85\,243}{59\,875\,200}x^{3} \\
&&+\frac{858\,623}{59\,875\,200}x^{2}+\frac{15\,950\,191}{119\,750\,400}x+%
\frac{14\,556\,793}{17\,107\,200}>0,
\end{eqnarray*}%
which proves the second one, and the proof is complete.
\end{proof}

The following lemma offers a simple criterion to determine the sign of a\
class of special polynomial on given interval contained in $\left( 0,\infty
\right) $ without using Descartes' Rule of Signs, which play an important
role in studying for certain special functions, see for example \cite%
{Yang-AAA-2014-702718}, \cite{Yand-JMI-ip}. A series version can be found in 
\cite{Yang-arxiv-1705-05704}.

\begin{lemma}[{\protect\cite[Lemma 7]{Yang-AAA-2014-702718}}]
\label{L-pz}Let $n\in \mathbb{N}$ and $m\in \mathbb{N}\cup \{0\}$ with $n>m$
and let $P_{n}\left( t\right) $ be an $n$ degrees polynomial defined by 
\begin{equation}
P_{n}\left( t\right) =\sum_{i=m+1}^{n}a_{i}t^{i}-\sum_{i=0}^{m}a_{i}t^{i},
\label{2.4}
\end{equation}%
where $a_{n},a_{m}>0$, $a_{i}\geq 0$ for $0\leq i\leq n-1$ with $i\neq m$.
Then there is a unique number $t_{m+1}\in \left( 0,\infty \right) $ to
satisfy $P_{n}\left( t\right) =0$ such that $P_{n}\left( t\right) <0$ for $%
t\in \left( 0,t_{m+1}\right) $ and $P_{n}\left( t\right) >0$ for $t\in
\left( t_{m+1},\infty \right) $.

Consequently, for given $t_{0}>0$, if $P_{n}\left( t_{0}\right) >0$ then $%
P_{n}\left( t\right) >0$ for $t\in \left( t_{0},\infty \right) $ and if $%
P_{n}\left( t_{0}\right) <0$ then $P_{n}\left( t\right) <0$ for $t\in \left(
0,t_{0}\right) $.
\end{lemma}

\section{Proof of Theorem \protect\ref{MT-1}}

With the aid of lemmas collected in Section 2, we can prove Theorem \ref%
{MT-1}.

\begin{proof}[Proof of Theorem \protect\ref{MT-1}]
Differentiation yields%
\begin{eqnarray*}
f_{0}^{\prime }\left( x\right) &=&\psi \left( x+1\right) -\frac{1}{2}\ln
\left( x\sinh \frac{1}{x}\right) +\frac{1}{2x}\coth \frac{1}{x} \\
&&-\ln x-\frac{1}{2x}-\frac{1}{2}+\frac{7}{324}\frac{175x^{2}+99}{%
x^{4}\left( 35x^{2}+33\right) ^{2}}, \\
f_{0}^{\prime \prime }\left( x\right) &=&\psi ^{\prime }\left( x+1\right) +%
\frac{1}{2x^{3}}\frac{1}{\sinh ^{2}\left( 1/x\right) } \\
&&-\frac{3}{2x}+\frac{1}{2x^{2}}-\frac{7}{54}\frac{6125x^{4}+6545x^{2}+2178}{%
x^{5}\left( 35x^{2}+33\right) ^{3}}.
\end{eqnarray*}

Since $\lim_{x\rightarrow \infty }f_{0}\left( x\right) =\lim_{x\rightarrow
\infty }f_{0}^{\prime }\left( x\right) =0$, it suffices to prove $%
f_{0}^{\prime \prime }\left( x\right) >0$ for $x\geq 1$. Replacing $x$ by $%
\left( x+1/2\right) $ in the right hand side inequality of (\ref{L-dpsi>})
leads to%
\begin{equation*}
\psi ^{\prime }\left( x+1\right) >\frac{7}{30}\frac{\left( 2x+1\right)
\left( 165x^{4}+330x^{3}+815x^{2}+650x+417\right) }{%
77x^{6}+231x^{5}+560x^{4}+735x^{3}+623x^{2}+294x+60},
\end{equation*}%
which indicates that%
\begin{eqnarray*}
f_{0}^{\prime \prime }\left( x\right) &>&\frac{7}{30}\frac{\left(
2x+1\right) \left( 165x^{4}+330x^{3}+815x^{2}+650x+417\right) }{%
77x^{6}+231x^{5}+560x^{4}+735x^{3}+623x^{2}+294x+60}+\frac{1}{2x^{3}}\frac{1%
}{\sinh ^{2}\left( 1/x\right) } \\
&&-\frac{3}{2x}+\frac{1}{2x^{2}}-\frac{7}{54}\frac{6125x^{4}+6545x^{2}+2178}{%
x^{5}\left( 35x^{2}+33\right) ^{3}}:=f_{01}\left( \frac{1}{x}\right) .
\end{eqnarray*}%
Arranging gives%
\begin{eqnarray*}
f_{01}\left( t\right) &=&\frac{t}{2}\left( \frac{t}{\sinh t}\right) ^{2}+%
\frac{7}{30}\frac{t\left( t+2\right) \left(
417t^{4}+650t^{3}+815t^{2}+330t+165\right) }{%
60t^{6}+294t^{5}+623t^{4}+735t^{3}+560t^{2}+231t+77} \\
&&-\frac{3}{2}t+\frac{1}{2}t^{2}-\frac{7}{54}t^{7}\frac{%
2178t^{4}+6545t^{2}+6125}{\left( 33t^{2}+35\right) ^{3}}.
\end{eqnarray*}%
where $t=1/x\in \left( 0,1\right) $. Applying the first inequality of (\ref%
{t/sht>}) we have%
\begin{eqnarray*}
f_{01}\left( t\right) &>&\frac{t}{2}\left( 1-\frac{1}{6}t^{2}+\frac{7}{360}%
t^{4}-\frac{31}{15\,120}t^{6}+\frac{127}{604\,800}t^{8}-\frac{73}{3421\,440}%
t^{10}\right) ^{2} \\
&&+\frac{7}{30}\frac{t\left( t+2\right) \left(
417t^{4}+650t^{3}+815t^{2}+330t+165\right) }{%
60t^{6}+294t^{5}+623t^{4}+735t^{3}+560t^{2}+231t+77} \\
&&-\frac{3}{2}t+\frac{1}{2}t^{2}-\frac{7}{54}t^{7}\frac{%
2178t^{4}+6545t^{2}+6125}{\left( 33t^{2}+35\right) ^{3}} \\
&=&\frac{t^{11}\times p_{22}\left( t\right) }{\left( 33t^{2}+35\right)
^{3}\left( 60t^{6}+294t^{5}+623t^{4}+735t^{3}+560t^{2}+231t+77\right) },
\end{eqnarray*}%
where $p_{22}\left( t\right) =\sum_{k=0}^{22}a_{k}t^{k}$ with%
\begin{equation*}
\begin{array}{ccccc}
a_{22} & a_{21} & a_{20} & a_{19} & a_{18} \\ 
\frac{58\,619}{119\,439\,360} & \frac{2872\,331}{1194\,393\,600} & -\frac{%
150\,639\,953}{50\,164\,531\,200} & -\frac{402\,182\,039}{11\,943\,936\,000}
& \frac{214\,165\,238\,137}{6437\,781\,504\,000}%
\end{array}%
\end{equation*}%
\begin{equation*}
\begin{array}{cccc}
a_{17} & a_{16} & a_{15} & a_{14} \\ 
\frac{224\,320\,158\,179}{492\,687\,360\,000} & -\frac{1469\,516\,232\,022%
\,339}{4780\,052\,766\,720\,000} & -\frac{107\,829\,513\,340\,517}{%
19\,510\,419\,456\,000} & \frac{1227\,464\,630\,525\,327}{%
573\,606\,332\,006\,400}%
\end{array}%
\end{equation*}%
\begin{equation*}
\begin{array}{cccc}
a_{13} & a_{12} & a_{11} & a_{10} \\ 
\frac{40\,990\,762\,057\,313\,921}{682\,864\,680\,960\,000} & \frac{%
1859\,898\,503\,651\,431}{585\,312\,583\,680\,000} & -\frac{%
13\,202\,571\,814\,150\,457}{24\,831\,442\,944\,000} & -\frac{%
27\,685\,269\,148\,007\,477}{74\,494\,328\,832\,000}%
\end{array}%
\end{equation*}%
\begin{equation*}
\begin{array}{cccc}
a_{9} & a_{8} & a_{7} & a_{6} \\ 
\frac{734\,284\,235\,570\,623}{229\,920\,768\,000} & \frac{%
722\,576\,509\,559\,549}{344\,881\,152\,000} & -\frac{444\,392\,576\,792\,851%
}{19\,707\,494\,400} & -\frac{2348\,474\,362\,865\,491}{59\,122\,483\,200}%
\end{array}%
\end{equation*}%
\begin{equation*}
\begin{array}{cccccc}
a_{5} & a_{4} & a_{3} & a_{2} & a_{1} & a_{0} \\ 
\frac{1776\,198\,096\,757}{51\,321\,600} & -\frac{21\,774\,907\,040\,747}{%
615\,859\,200} & \frac{3740\,791\,861\,177}{13\,685\,760} & \frac{%
4592\,761\,525\,177}{41\,057\,280} & \frac{2341\,955}{9} & \frac{2341\,955}{%
27}%
\end{array}%
\end{equation*}%
It remains to prove $p_{22}\left( t\right) =\sum_{k=0}^{22}a_{k}t^{k}>0$ for 
$t\in (0,1]$. Since $a_{k}>0$ for $k=22$, $21$, $18$, $17$, $14$, $13$, $12$%
, $9$, $8$, $3$, $2$, $1$, $0$ and $a_{k}<0$ for $k=20$, $19$, $16$, $15$ $%
11 $, $10$, $7$, $6$, $4$, we have%
\begin{equation*}
p_{22}\left( t\right)
=\sum_{k=0}^{22}a_{k}t^{k}=\sum_{a_{k}>0}a_{k}t^{k}+%
\sum_{a_{k}<0}a_{k}t^{k}>\sum_{k=20,19,16,15,11,10,7,6,4}a_{k}t^{k}+%
\sum_{k=0}^{3}a_{k}t^{k}:=p_{20}\left( t\right) .
\end{equation*}%
Clearly, the coefficients of the polynomial $-p_{20}\left( t\right) $
satisfy the conditions for Lemma \ref{L-pz}, and%
\begin{equation*}
-p_{20}\left( 1\right) =\sum_{k=20,19,16,15,11,10,7,6,4}\left( -a_{k}\right)
-\sum_{k=0}^{3}a_{k}=-\frac{1135\,768\,202\,621\,781\,774\,901}{%
1792\,519\,787\,520\,000}<0.
\end{equation*}%
It then follows that $p_{20}\left( t\right) >0$ for $t\in (0,1]$, and so is $%
p_{22}\left( t\right) $, which implies $f_{01}\left( t\right) >0$ for $t\in
(0,1]$, Consequently, we thereby deduce that $f_{0}^{\prime \prime }\left(
x\right) >0$ for all $x\geq 1$. This completes the proof.
\end{proof}

As a direct consequence of Theorem \ref{MT-1}, we immediately get the
following

\begin{corollary}
\label{C-1}For $n\in \mathbb{N}$, the double inequality%
\begin{equation*}
\exp \frac{7}{324n^{3}\left( 35n^{2}+33\right) }<\frac{n!}{\sqrt{2\pi n}%
\left( n/e\right) ^{n}\left( n\sinh n^{-1}\right) ^{n/2}}<\lambda \exp \frac{%
7}{324n^{3}\left( 35n^{2}+33\right) }
\end{equation*}%
holds with the best constant%
\begin{equation*}
\lambda =\exp f_{0}\left( 1\right) =\frac{1}{\sqrt{2\pi \sinh 1}}\exp \frac{%
22\,025}{22\,032}\approx 1.000\,024\,067\text{.}
\end{equation*}
\end{corollary}

Denote by%
\begin{equation*}
D_{0}\left( y\right) =y-\ln \left( 1+y\right) \text{, \ }y=\dfrac{7}{%
324x^{3}\left( 35x^{2}+33\right) }.
\end{equation*}%
Then it is easy to check that for $x>1$,%
\begin{equation*}
\frac{dD_{0}\left( y\right) }{dx}=-\tfrac{49}{324}\tfrac{175x^{2}+99}{%
x^{4}\left( 35x^{2}+33\right) ^{2}\left( 11\,340x^{5}+10\,692x^{3}+7\right) }%
<0,
\end{equation*}%
\begin{equation*}
\frac{d^{2}D_{0}\left( y\right) }{dx^{2}}=\tfrac{343}{54}\tfrac{\left(
18\,191\,250x^{9}+37\,110\,150x^{7}+24\,992\,550x^{5}+6125x^{4}+5821%
\,794x^{3}+6545x^{2}+2178\right) }{x^{5}\left( 35x^{2}+33\right) ^{3}\left(
11\,340x^{5}+10\,692x^{3}+7\right) ^{2}}>0.
\end{equation*}%
That is to say, $x\mapsto D_{0}\left( y\right) $ is decreasing and convex on 
$\left( 1,\infty \right) $. Theorem \ref{MT-1} together with this yields
that so is $f_{0}^{\ast }\left( x\right) :=f_{0}\left( x\right) +D_{0}\left(
y\right) $ on $\left( 1,\infty \right) $.

\begin{corollary}
\label{C-2}The function%
\begin{equation*}
f_{0}^{\ast }\left( x\right) =\ln \Gamma \left( x+1\right) -\ln \sqrt{2\pi }%
-\left( x+\frac{1}{2}\right) \ln x+x-\frac{x}{2}\ln \left( x\sinh \frac{1}{x}%
\right) -\ln \left( 1+\frac{7}{324x^{3}\left( 35x^{2}+33\right) }\right)
\end{equation*}%
is also strictly decreasing and convex from $\left( 1,\infty \right) $ onto $%
\left( 0,f_{0}^{\ast }\left( 1\right) \right) $, where%
\begin{equation*}
f_{0}^{\ast }\left( 1\right) =1-\ln \frac{22\,039}{22\,032}-\ln \sqrt{2\pi
\sinh 1}\approx 0.00002412.
\end{equation*}
\end{corollary}

\begin{remark}
Corollary \ref{C-2} offers another approximation formula%
\begin{equation}
\Gamma \left( x+1\right) \thicksim \sqrt{2\pi x}\left( \dfrac{x}{e}\right)
^{x}\left( x\sinh \frac{1}{x}\right) ^{x/2}\left( 1+\frac{7}{324}\frac{1}{%
x^{3}\left( 35x^{2}+33\right) }\right) =W_{2}^{\ast }\left( x\right) .
\label{W2*}
\end{equation}%
Also, for $n\in \mathbb{N}$, it holds that%
\begin{equation*}
1+\frac{7}{324n^{3}\left( 35n^{2}+33\right) }<\frac{n!}{\sqrt{2\pi n}\left(
n/e\right) ^{n}\left( n\sinh n^{-1}\right) ^{n/2}}<\lambda ^{\ast }\left( 1+%
\frac{7}{324n^{3}\left( 35n^{2}+33\right) }\right)
\end{equation*}%
with the best constant%
\begin{equation*}
\lambda ^{\ast }=\exp f_{0}^{\ast }\left( 1\right) =\frac{22\,032}{22\,039}%
\frac{e}{\sqrt{2\pi \sinh 1}}\approx 1.000\,024117\text{.}
\end{equation*}
\end{remark}

\section{Numerical comparisons}

It is known that an excellent approximation for gamma function is fairly
accurate but relatively simple. In view of the above, we list some known
approximation formulas for gamma function and compare them with $W_{1}\left(
x\right) $ given by (\ref{W1}) and our new one $W_{2}\left( x\right) $
defined by (\ref{W2}) in this section.

It has been shown in \cite{Chen-JNT-164-2016} that as $x\rightarrow \infty $%
, Ramanujan's \cite[P. 339]{Ramanujan-SB-1988} approximation formula%
\begin{equation*}
\Gamma \left( x+1\right) \thicksim \sqrt{\pi }\left( \dfrac{x}{e}\right)
^{x}\left( 8x^{3}+4x^{2}+x+\dfrac{1}{30}\right) ^{1/6}\left( 1+O\left( \frac{%
1}{x^{4}}\right) \right) :=R\left( x\right) ,
\end{equation*}%
Smith's one \cite[Eq. (42)]{Smith-2006}%
\begin{equation*}
\Gamma \left( x+\frac{1}{2}\right) \thicksim \sqrt{2\pi }\left( \dfrac{x}{e}%
\right) ^{x}\left( 2x\tanh \frac{1}{2x}\right) ^{x/2}\left( 1+O\left( \frac{1%
}{x^{5}}\right) \right) :=S\left( x\right) ,
\end{equation*}%
Nemes' one \cite[Corollary 4.1]{Nemes-AM-95-2010}%
\begin{equation*}
\Gamma \left( x+1\right) \thicksim \sqrt{2\pi x}\left( \dfrac{x}{e}\right)
^{x}\left( 1+\frac{1}{12x^{2}-1/10}\right) ^{x}\left( 1+O\left( \frac{1}{%
x^{5}}\right) \right) =:N_{1}\left( x\right) ,
\end{equation*}%
Chen's one \cite{Chen-JNT-164-2016}%
\begin{equation}
\Gamma (x+1)\thicksim \sqrt{2\pi x}\left( \dfrac{x}{e}\right) ^{x}\left( 1+%
\dfrac{1}{12x^{3}+24x/7-1/2}\right) ^{x^{2}+53/210}\left( 1+O\left( \frac{1}{%
x^{7}}\right) \right) :=C\left( x\right) .  \label{C}
\end{equation}%
Moreover, it is easy to check that Nemes' \cite{http/gamma} another one%
\begin{equation}
\Gamma \left( x+1\right) \thicksim \sqrt{2\pi x}\left( \dfrac{x}{e}\right)
^{x}\exp \left( \frac{210x^{2}+53}{360x\left( 7x^{2}+2\right) }\right)
\left( 1+O\left( \frac{1}{x^{7}}\right) \right) :=N_{2}\left( x\right) ,
\label{N2}
\end{equation}%
Yand and Chu's \cite[Propositions 4 and 5]{Yang-AMC-270-2015} ones%
\begin{eqnarray*}
\Gamma \left( x+\frac{1}{2}\right) &=&\sqrt{2\pi }\left( \dfrac{x}{e}\right)
^{x}\exp \left( -\dfrac{1}{24}\dfrac{x}{x^{2}+7/120}\right) \left( 1+O\left( 
\frac{1}{x^{5}}\right) \right) :=Y_{1}\left( x\right) , \\
\Gamma \left( x+\frac{1}{2}\right) &=&\sqrt{2\pi }\left( \dfrac{x}{e}\right)
^{x}\exp \left( -\frac{1}{24x}+\frac{7}{2880x}\frac{1}{x^{2}+31/98}\right)
\left( 1+O\left( \frac{1}{x^{7}}\right) \right) :=Y_{2}\left( x\right) ,
\end{eqnarray*}%
Windschitl one \cite{http/gamma}%
\begin{equation*}
\Gamma (x+1)\thicksim \sqrt{2\pi x}\left( \dfrac{x}{e}\right) ^{x}\left(
x\sinh \frac{1}{x}+\frac{1}{810x^{6}}\right) ^{x/2}\left( 1+O\left( \frac{1}{%
x^{7}}\right) \right) =W_{1}\left( x\right) .
\end{equation*}%
For our new ones $W_{2}\left( x\right) $ given in (\ref{W2}) and its
counterpart $W_{2}^{\ast }\left( x\right) $ given in (\ref{W2*}), we easily
check that%
\begin{equation*}
\lim_{x\rightarrow \infty }\frac{\ln \Gamma \left( x+1\right) -\ln
W_{2}\left( x\right) }{x^{-9}}=\lim_{x\rightarrow \infty }\frac{\ln \Gamma
\left( x+1\right) -\ln W_{2}^{\ast }\left( x\right) }{x^{-9}}=\frac{869}{%
2976\,750},
\end{equation*}%
which show that the rate of $W_{2}\left( x\right) $ and $W_{2}^{\ast }\left(
x\right) $ converging to $\Gamma \left( x+1\right) $ are as $x^{-9}$.

From these, we see that our new Windschitl type approximation formulas $%
W_{2}\left( x\right) $ and $W_{2}^{\ast }\left( x\right) $ are the best
among ones listed above, which can be also seen from the following Table 1.%
\begin{equation*}
\text{Table 1: Comparison among }N_{2}\text{ (\ref{N2}), }C\text{ (\ref{C}), 
}W_{1}\text{ (\ref{W1}) and }W_{2}\text{ (\ref{W2})}
\end{equation*}%
\begin{equation*}
\begin{tabular}{|l|l|l|l|l|}
\hline
$x$ & $\left\vert \frac{N_{2}\left( x\right) -\Gamma \left( x+1\right) }{%
\Gamma \left( x+1\right) }\right\vert $ & $\left\vert \frac{C\left( x\right)
-\Gamma \left( x+1\right) }{\Gamma \left( x+1\right) }\right\vert $ & $%
\left\vert \frac{W_{1}\left( x\right) -\Gamma \left( x+1\right) }{\Gamma
\left( x+1\right) }\right\vert $ & $\left\vert \frac{W_{2}\left( x\right)
-\Gamma \left( x+1\right) }{\Gamma \left( x+1\right) }\right\vert $ \\ \hline
$1$ & $1.114\times 10^{-4}$ & $1.398\times 10^{-4}$ & $1.832\times 10^{-4}$
& $2.407\times 10^{-5}$ \\ \hline
$2$ & $1.900\times 10^{-6}$ & $2.222\times 10^{-6}$ & $2.668\times 10^{-6}$
& $2.308\times 10^{-7}$ \\ \hline
$5$ & $4.353\times 10^{-9}$ & $4.956\times 10^{-9}$ & $5.743\times 10^{-9}$
& $1.249\times 10^{-10}$ \\ \hline
$10$ & $3.609\times 10^{-11}$ & $4.088\times 10^{-11}$ & $4.710\times
10^{-11}$ & $2.785\times 10^{-13}$ \\ \hline
$20$ & $2.864\times 10^{-13}$ & $3.240\times 10^{-13}$ & $3.727\times
10^{-13}$ & $5.634\times 10^{-16}$ \\ \hline
$50$ & $4.713\times 10^{-16}$ & $5.330\times 10^{-16}$ & $6.129\times
10^{-16}$ & $1.492\times 10^{-19}$ \\ \hline
$100$ & $3.684\times 10^{-18}$ & $4.166\times 10^{-18}$ & $4.791\times
10^{-18}$ & $2.918\times 10^{-22}$ \\ \hline
\end{tabular}%
\end{equation*}

\section{Acknowledgements}

The authors would like to express their sincere thanks to the editors and
reviewers for their great efforts to improve this paper.

This work was supported by the Fundamental Research Funds for the Central
Universities (No. 2015ZD29) and the Higher School Science Research Funds of
Hebei Province of China (No. Z2015137).

\section{Competing interests}

The authors declare that they have no competing interests.

\section{Authors' contributions}

All authors contributed equally to the writing of this paper. All authors
read and approved the final manuscript.

\end{document}